\documentclass[10pt]{amsart}
\usepackage{amssymb, amsthm, amsmath}
\usepackage[all]{xy}
\usepackage{graphicx}

\newtheorem{lemma}{Lemma}
\theoremstyle{definition}

\newtheorem{example}{Example}

\newcommand\A{{\mathbb A}}

\newcommand\Z{{\mathbb Z}}

\newcommand\al{\alpha}
\newcommand{\be}{\beta}
\newcommand\la{\lambda}

\newcommand\ssm{\smallsetminus}

\newcommand\noin{\noindent}

\newcommand\eqto{\stackrel{\lower1.5pt\hbox{$\scriptstyle\sim\,$}}\to}

\begin{document}

\title[The theory of Schur polynomials revisited]
{The theory of Schur polynomials revisited}

\date{December 10, 2011}

\author{Harry~Tamvakis} \address{University of Maryland, Department of
Mathematics, 1301 Mathematics Building, College Park, MD 20742, USA}
\email{harryt@math.umd.edu}

\subjclass[2000]{Primary 05E05; Secondary 14N15}

\thanks{The author was supported in part by NSF Grant DMS-0901341.}

\begin{abstract}
We use Young's raising operators to give short and uniform proofs of
several well known results about Schur polynomials and symmetric
functions, starting from the Jacobi-Trudi identity.
\end{abstract}

\maketitle

\section{Introduction}

One of the earliest papers to study the symmetric functions later
known as the Schur polynomials $s_\la$ is that of Jacobi \cite{J},
where the following two formulas are found. The first is Cauchy's
definition of $s_\la$ as a quotient of determinants:
\begin{equation}
\label{def1}
s_\la(x_1,\ldots,x_n) =
\left.\det(x_i^{\la_i+n-j})_{i,j}\right\slash\det(x_i^{n-j})_{i,j}
\end{equation}
where $\la=(\la_1,\ldots,\la_n)$ is an integer partition with at most
$n$ non-zero parts. The second is the Jacobi-Trudi identity
\begin{equation}
\label{def2}
s_\la = \det(h_{\la_i+j-i})_{1\leq i,j \leq n}
\end{equation}
which expresses $s_\la$ as a polynomial in the complete symmetric
functions $h_r$, $r\geq 0$. Nearly a century later, Littlewood
\cite{L} obtained the positive combinatorial expansion
\begin{equation}
\label{def3}
s_\la(x) = \sum_T x^{c(T)}
\end{equation}
where the sum is over all semistandard Young tableaux $T$ of shape
$\la$, and $c(T)$ denotes the content vector of $T$.

The traditional approach to the theory of Schur polynomials begins
with the classical definition (\ref{def1}); see for example \cite{FH,
M, Ma}. Since equation (\ref{def1}) is a special case of the Weyl
character formula, this method is particularly suitable for
applications to representation theory. The more combinatorial
treatments \cite{Sa, Sta} use (\ref{def3}) as the definition of
$s_\la(x)$, and proceed from there. It is not hard to relate formulas
(\ref{def1}) and (\ref{def3}) to each other directly; see e.g.\
\cite{Pr, Ste}. 

In this article, we take the Jacobi-Trudi formula (\ref{def2}) as the
starting point, where the $h_r$ represent algebraically independent
variables. We avoid the use of the $x$ variables or `alphabets' and
try to prove as much as we can without them. For this purpose, it
turns out to be very useful to express (\ref{def2}) in the alternative
form
\begin{equation}
\label{def2r}
s_\la = \prod_{i<j} (1-R_{ij}) \, h_\la
\end{equation}
where the $R_{ij}$ are Young's raising operators \cite{Y} and $h_\la =
h_{\la_1}h_{\la_2}\cdots h_{\la_n}$. The equivalence of (\ref{def2}) and 
(\ref{def2r}) follows immediately from the Vandermonde identity. 

The motivation for this approach to the subject comes from Schubert
calculus. It is well known that the algebra of Schur polynomials
agrees with that of the Schubert classes in the cohomology ring of the
complex Grassmannian $G(k,r)$, when $k$ and $r$ are sufficiently
large. Giambelli \cite{G} showed that the Schubert classes on $G(k,r)$
satisfy the determinantal formula (\ref{def2}); the closely related
Pieri rule \cite{P} had been obtained geometrically a few years
earlier.  Recently, with Buch and Kresch \cite{BKT1, BKT2}, we proved
analogues of the Pieri and Giambelli formulas for the isotropic
Grassmannians which are quotients of the symplectic and orthogonal
groups. Our Giambelli formulas for the Schubert classes on these
spaces are not determinantal, but rather are stated in terms of
raising operators. In \cite{T}, we used raising operators to obtain a
tableau formula for the corresponding theta polynomials, which is an
analogue of Littlewood's equation (\ref{def3}) in this
context. Moreover, the same methods were applied in loc.\ cit.\ to
provide new proofs of similar facts about the Hall-Littlewood functions.

Our aim here is to give a self-contained treatment of those aspects of
the theory of Schur polynomials and symmetric functions which follow
naturally from the above raising operator approach. Using
(\ref{def2r}) as the definition of Schur polynomials, we give short
proofs of the Pieri and Littlewood-Richardson rules, and follow this
with a discussion -- in this setting -- of the duality involution,
Cauchy identities, and skew Schur polynomials. We next introduce the
variables $x=(x_1,x_2,\ldots)$ and study the ring $\Lambda$ of
symmetric functions in $x$ from scratch.  In particular, we derive the
bialternant and tableau formulas (\ref{def1}) and (\ref{def3}) for
$s_\la(x)$. See \cite{La} for an approach to these topics which begins
with (\ref{def2}) but is based on alphabets and properties of
determinants such as the Binet-Cauchy formula, and \cite{vL, Ste} for
a different treatment which employs alternating sums stemming from
(\ref{def1}).

Most of the proofs in this article are streamlined versions of more
involved arguments contained in \cite{BKT2}, \cite{M}, and
\cite{T}. The proof we give of the Littlewood-Richardson rule from the
Pieri rule is essentially that of Remmel-Shimozono \cite{RS} and
Gasharov \cite{G}, but expressed in the concise form adapted by
Stembridge \cite{Ste}. Each of these proofs employs the same sign
reversing involution on a certain set of Young tableaux, which
originates in the work of Berenstein-Zelevinsky \cite{BZ}. The version
given here does not use formulas (\ref{def1}) and (\ref{def3}) at all,
but relies on the alternating property of the determinant
(\ref{def2}), which serves the same purpose.

The reduction formula (\ref{reduction}) for the number of variables in
$s_\la(x_1,\ldots,x_n)$ is classically known as a `branching rule' for
the characters of the general linear group \cite{Pr, W}. Our
terminology differs because there are similar results in
situations where the connection with representation theory is not
available (see \cite{T}). We use the reduction formula to derive
(\ref{def3}) from (\ref{def2r}); a different cancellation argument
relating formulas (\ref{def2}) and (\ref{def3}) to each other is due
to Gessel-Viennot \cite{GV, Sa}.

We find that the short arguments in this article are quite uniform,
especially when compared to other treatments of the same material. On
the other hand, much of the theory of Schur polynomials does not
readily fit into the present framework. Missing from the discussion
are the Hall inner product, the Hopf algebra structure on $\Lambda$,
the basis of power sums, the character theory of the symmetric and
general linear groups, Young tableau algorithms such as jeu de taquin,
the plactic algebra, and noncommutative symmetric functions. These
topics and many more can be added following standard references such
as \cite{F, La, M, Ma, Sa, Sta, Z}, but are not as natural from the
point of view adopted here, which stems from Grassmannian Schubert
calculus. A similar approach may be used to study the theory of Schur
$Q$-polynomials and more generally of Hall-Littlewood functions; some
of this story may be found in \cite{T}.

The author is indebted to his collaborators Anders
Buch and Andrew Kresch for their efforts on the related projects
\cite{BKT1, BKT2}.

\section{The algebra of Schur polynomials}
\label{typeA}

\subsection{Preliminaries} An {\em integer sequence} or 
{\em integer vector} is a sequence of integers
$\al=(\al_1,\al_2,\ldots)$ with only finitely $\al_i$ non-zero. The
{\em length} of $\al$, denoted $\ell(\al)$, is largest integer
$\ell\geq 0$ such that $\al_\ell\neq 0$. We identify an integer
sequence of length $\ell$ with the vector consisting of its first
$\ell$ terms.  We let $|\al| = \sum \al_i$ and write $\al\geq \be$ if
$\al_i \geq \be_i$ for each $i$.  An integer sequence $\al$ is a {\em
composition} if $\al_i\geq 0$ for all $i$ and a {\em partition} if
$\al_i \geq \al_{i+1}\geq 0$ for all $i$.

Consider the polynomial ring $\A=\Z[u_1,u_2,\ldots]$ where the $u_i$
are countably infinite commuting independent variables. We regard $\A$
as a graded ring with each $u_i$ having graded degree $i$, and
adopt the convention here and throughout the paper that $u_0=1$ while
$u_r=0$ for $r<0$. For each integer vector $\al$, set $u_{\al} =
\prod_iu_{\al_i}$; then $\A$ has a free $\Z$-basis consisting of the
monomials $u_{\la}$ for all partitions $\la$.

For two integer sequences $\al$, $\be$ such that $|\al|= |\be|$, we say 
that $\al$ {\em dominates} $\be$
and write $\al\succeq \be$ if $\al_1+\cdots + \al_i\geq \be_1+\cdots
+\be_i$ for each $i$.  Given any integer sequence
$\alpha=(\alpha_1,\alpha_2,\ldots)$ and $i<j$, we define
\[
R_{ij}(\alpha) = (\alpha_1,\ldots,\alpha_i+1,\ldots,\alpha_j-1,
\ldots).
\] 
A {\em raising operator} $R$ is any monomial in these $R_{ij}$'s.
Note that we have $R\,\al\succeq \al$ for all integer sequences $\al$.
For any raising operator $R$, define $R\,u_{\al} = u_{R\al}$. Here the
operator $R$ acts on the index $\al$, and not on the monomial $u_\al$
itself. Thus, if the components of $\al$ are a permutation of the
components of $\be$, then $u_\al = u_\be$ as elements of $\A$, but it
may happen that $R\,u_{\al} \neq R\, u_{\be}$.  Formal manipulations
using these raising operators are justified carefully in the
following section.  Note that if $\al_\ell<0$ for $\ell=\ell(\al)$,
then $R\,u_\al=0$ in $\A$ for any raising operator $R$.

\subsection{Schur polynomials}
For any integer vector $\al$, define the {\em Schur polynomial}
$U_{\al}$ by the formula
\begin{equation}
\label{giambelliA}
U_{\al} := \prod_{i<j} (1-R_{ij})\, u_{\al}.
\end{equation}
Although the product in (\ref{giambelliA}) is infinite, if we expand
it into a formal series we find that only finitely many of the summands
are nonzero; hence, $U_{\al}$ is well defined. We will show that
equation (\ref{giambelliA}) may be written in the determinantal form
\begin{equation}
\label{giambelliA2}
U_{\al} = \det(u_{\al_i+j-i})_{1\leq i,j\leq \ell} =
\sum_{w\in S_\ell}(-1)^w u_{w(\al+\rho_\ell)-\rho_\ell}
\end{equation}
where $\ell$ denotes the length of $\al$ and $\rho_\ell =
(\ell-1,\ell-2,\ldots,1,0)$.

Algebraic expressions and identities involving raising operators like
the above can be justified by viewing them as the image of a
$\Z$-linear map $\Z[\Z^\ell]\to\A$, where $\Z[\Z^\ell]$ denotes the
group algebra of $(\Z^\ell,+)$. We let $x_1,\ldots,x_\ell$ be
independent variables and identify $\Z[\Z^\ell]$ with
$\Z[x_1,x_1^{-1},\ldots,x_\ell,x_\ell^{-1}]$. For any integer vector
$\al=(\al_1,\ldots,\al_\ell)$ and raising operator $R$, set $x^\al =
x_1^{\al_1}\cdots x_\ell^{\al_\ell}$ and $R\,x^\al = x^{R\al}$. Then if
$\psi:\Z[\Z^\ell]\to\A$ is the $\Z$-linear map determined by
$\psi(x^\al) = u_\al$ for each $\al$, we have $R\, u_\al =
\psi(x^{R\al})$. It follows from the Vandermonde identity
\[
\prod_{1\leq i < j \leq \ell}(x_j - x_i) =
\det (x_i^{j-1})_{1\leq i,j\leq \ell} 
\]
that 
\[
\prod_{1\leq i<j\leq \ell} (1-R_{ij})\, x^{\al} = 
\prod_{1\leq i<j\leq \ell} (1-x_ix_j^{-1})\, x^{\al} = 
\det (x_i^{\al_i+j-i})_{1\leq i,j\leq \ell}.
\]
Now apply the map $\psi$ to both ends of the above
equation to obtain (\ref{giambelliA2}).

\begin{example}
We have
\begin{gather*}
U_{(5,4,2)}=(1-R_{12})(1-R_{13})(1-R_{23})\, u_{(5,4,2)} \\
= (1-R_{12}-R_{13}-R_{23}+R_{12}R_{13}+R_{12}R_{23}+R_{13}R_{23}-R_{12}R_{13}R_{23})
\, u_{(5,4,2)} \\
= u_{(5,4,2)}-u_{(6,3,2)}-u_{(6,4,1)}-u_{(5,5,1)}+u_{(7,3,1)}+u_{(6,4,1)} 
+u_{(6,5,0)}-u_{(7,4,0)}
 \\ 
=u_5u_4u_2-u_6u_3u_2-u_5^2u_1+u_7u_3u_1+u_6u_5-u_7u_4
= \left|\begin{array}{ccc}
u_5 & u_6 & u_7 \\ u_3 & u_4 & u_5 \\ 1 & u_1 & u_2
\end{array}\right|.
\end{gather*}
\end{example}

If $\al=(\al_1,\ldots,\al_\ell)$ and $\be=(\be_1,\ldots,\be_m)$ are
two integer vectors and $r,s\in \Z$, we let $(\al,r,s,\be)$ denote the
integer vector $(\al_1,\ldots,\al_\ell, r, s,
\be_1,\ldots,\be_m)$. The next lemma is known as a `straightening law'
for the $U_\al$.

\begin{lemma}
\label{altlem}
{\em (a)} Let $\al$ and $\be$ be integer vectors.  
Then for any $r,s\in\Z$ we have
\[ U_{(\al,r,s,\be)} = - U_{(\al,s-1,r+1,\be)} \,. \]

\medskip
\noin {\em (b)} Let $\al=(\al_1,\ldots,\al_\ell)$ be any integer
vector.  Then $U_\al=0$ unless $\al+\rho_\ell = w(\mu+\rho_\ell)$ for
a (unique) permutation $w\in S_\ell$ and partition $\mu$. In the
latter case, we have $U_\al = (-1)^w U_\mu$.
\end{lemma}
\begin{proof}
Both parts follow immediately from (\ref{giambelliA2}) and the
alternating property of the determinant.
\end{proof}

If $\la$ is any partition, clearly (\ref{giambelliA}) implies that
$U_{\la} = u_{\la} + \sum_{\mu\succ\la}a_{\la\mu} u_\mu$ where
$a_{\la\mu}\in\Z$ and the sum is over partitions $\mu$ which strictly
dominate $\la$. We deduce that the $U_{\la}$ for $\la$ a partition
form another $\Z$-basis of $\A$.

\subsection{Mirror identities}

 We will represent a partition $\la$ by its Young diagram of boxes,
 arranged in left-justified rows, with $\la_i$ boxes in row $i$. We
 write $\la\subset\mu$ instead of $\la\leq\mu$ for the containment
 relation between two Young diagrams; in this case the set-theoretic
 difference $\mu\ssm\la$ is the skew diagram $\mu/\la$. A skew diagram
 is a {\em horizontal} (resp.\ {\em vertical}) {\em strip} if it does
 not contain two boxes in the same column (resp.\ row). We write $\la
 \xrightarrow{p} \mu$ if $\mu/\la$ is a horizontal strip with $p$
 boxes.

\begin{lemma}
\label{mirrors}
Let $\la$ be a partition and $p\geq 0$ be an integer. Then we have
\begin{equation}
\label{mrs}
\sum_{\al\geq 0,\ |\al|=p} U_{\la+\al} = \sum_{\la\xrightarrow{p}\mu}U_\mu
\quad \text{ and } \quad
\sum_{\al\geq 0,\ |\al|=p} U_{\la-\al} = \sum_{\mu\xrightarrow{p}\la}U_\mu
\end{equation}
where the sums are over compositions $\al\geq 0$ with $|\al|=p$ and
partitions $\mu\supset\la$ (respectively $\mu\subset\la$) such that
$\la\xrightarrow{p}\mu$ (respectively, $\mu\xrightarrow{p}\la$).
Moreover, for every $n\geq \ell(\la)$, the identities (\ref{mrs})
remain true if the sums are taken over $\al$ and $\mu$ of length at
most $n$.
\end{lemma}
\begin{proof}
The proofs of the two identities are very similar, so we will only
discuss the second. Let us rewrite the sum $\sum_{\al\geq 0}
U_{\la-\al}$ as $\sum_{\nu\leq \la}U_\nu$, where the latter sum is
over integer sequences $\nu$ such that $\nu_i\leq \la_i$ for each
$i$ and $|\nu| = |\la|-p$.  Call any such sequence $\nu$ {\em bad} if
there exists a $j\geq 1$ such that $\nu_j < \la_{j+1}$, and let $X$ be
the set of all bad sequences. Define an involution $\iota:X\to X$ as
follows: for $\nu\in X$, choose $j$ minimal such that $\nu_j <
\la_{j+1}$, and set
\[
\iota(\nu) = (\nu_1,\ldots,\nu_{j-1},
\nu_{j+1}-1,\nu_j+1, \nu_{j+2},\ldots).
\]
Lemma \ref{altlem}(a) implies that $U_\nu + U_{\iota(\nu)} = 0$ for
every $\nu\in X$. Therefore all bad indices may be omitted from the
sum $\sum_{\nu\leq \la}U_\nu$, and this completes the
proof. Moreover, to evaluate $\sum_{\nu\leq \la}U_\nu$ in the
situation where $\nu_j=0$ for all $j>n$, notice that if the minimal
$j$ such that $\nu_j < \la_{j+1}$ is $j=n$, then $\nu_n < 0$ and
therefore $U_\nu=0$.
\end{proof}

\subsection{The Pieri rule} 
\label{pieriruleA}

For any $d\geq 1$ define the operator $R^d$ by
\[
R^d = \prod_{1\leq i<j \leq d} (1-R_{ij}).
\]
For $p>0$ and any partition $\la$ of length $\ell$, we compute 
\[
u_p\cdot U_\la = u_p\cdot R^{\ell}\, u_{\la} = R^{\ell}\, u_{(\la,p)}
= R^{\ell+1}\cdot \prod_{i=1}^\ell(1-R_{i,\ell+1})^{-1} \, u_{(\la,p)}
\]
\[
= R^{\ell+1}\cdot\prod_{i=1}^\ell(1+R_{i,\ell+1} + 
R_{i,\ell+1}^2 + \cdots)\,u_{(\la,p)} = \sum_{\al\geq 0} U_{\la+\al},
\]
where the sum is over all compositions $\al$ such that $|\al| = p$ and
$\al_j =0$ for $j > \ell+1$. Applying Lemma \ref{mirrors},
we arrive at the {\em Pieri rule}
\begin{equation}
\label{prule}
u_p\cdot U_\la = \sum_{\la\xrightarrow{p}\mu} U_\mu.
\end{equation}

Conversely, suppose that we are given a family $\{X_\la\}$ of elements
of $\A$, one for each partition $\la$, such that $X_p=u_p$ for every
integer $p\geq 0$ and the $X_\la$ satisfy the Pieri rule $X_p\cdot
X_\la = \sum_{\la\xrightarrow{p}\mu} X_\mu$. We claim then that
\[
X_\la = U_\la = \prod_{i<j}(1-R_{ij})\, u_\la
\]
for every partition $\la$. To see this, note that the Pieri rule
implies that
\begin{equation}
\label{itPrule}
U_\la + \sum_{\mu\succ\la} a_{\la\mu} \, U_\mu = 
u_{\la_1}\cdots u_{\la_\ell} = 
X_\la + \sum_{\mu \succ \la} a_{\la\mu}\, X_\mu
\end{equation}
for some constants $a_{\la\mu}\in\Z$. The claim now follows by induction
on $\la$. 

\begin{example} We have
\[
u_2\cdot U_{(3,3,1)} = U_{(5,3,1)}+U_{(4,3,2)}+U_{(4,3,1,1)}+
U_{(3,3,3)}+U_{(3,3,2,1)}.
\]
\end{example}

\subsection{Kostka numbers}
A {\em (semistandard) tableau} $T$ on the skew shape $\la/\mu$
is a filling of the boxes of $\la/\mu$ with positive integers, so that
the entries are weakly increasing along each row from left to right
and strictly increasing down each column. We can identify such a
tableau $T$ with a sequence of partitions
\[
\mu = \la^0 \xrightarrow{c_1} \la^1 \xrightarrow{c_2} \cdots
\xrightarrow{c_r}\la^r = \la
\]
such that for $1\leq i \leq r$ the horizontal strip $\la^i/\la^{i-1}$
consists of the $c_i$ boxes in $T$ with entry $i$. The composition
$c(T)=(c_1,\ldots,c_r)$ is called the {\em content} of $T$.

Let $\mu$ be a partition and $\al$ any integer vector. The equation 
\[
u_\al \,U_{\mu} = \sum_{\la} K_{\la/\mu,\al}\, U_\la
\]
summed over partitions $\la$ such that $\la\supset\mu$ defines the
{\em Kostka numbers} $K_{\la/\mu,\al}$. If $\al$ is not a composition
such that $|\al|=|\la/\mu|$ then we have
$K_{\la/\mu,\al}=0$. Otherwise, iteration of the Pieri rule shows that
$K_{\la/\mu,\al}$ equals the number of tableaux $T$ of shape $\la/\mu$
and content vector $c(T)=\al$. We deduce from equation (\ref{itPrule})
that the {\em Kostka matrix} $K=\{K_{\la,\mu}\}$, whose rows and columns
are indexed by partitions, is lower unitriangular with respect to the
dominance order.

\subsection{The Littlewood-Richardson rule}

Define the {\em Littlewood-Richardson
coefficients} to be the structure constants $c^\la_{\mu\nu}$ in the
equation
\begin{equation}
\label{scnsts}
U_\mu \cdot U_\nu = \sum_\la c^\la_{\mu\nu} \, U_\la.
\end{equation}
If $\ell=\ell(\nu)$, we compute that 
\begin{align*}
U_\mu \cdot U_\nu &= \sum_{w\in S_\ell} (-1)^w u_{w(\nu+\rho_\ell) - 
\rho_\ell} \, U_{\mu} \\
&= \sum_\la \sum_{w\in S_\ell} (-1)^w K_{\la/\mu,w(\nu+\rho_\ell) - 
\rho_\ell} \, U_\la
\end{align*}
from which we deduce that 
\begin{equation}
\label{ceq}
c_{\mu\nu}^{\la} = \sum_{(w,T)} (-1)^w
\end{equation}
where the sum is over all pairs $(w,T)$ such that $w\in S_\ell$ 
and $T$ is a tableau on $\la/\mu$ with $c(T)+\rho_\ell=w(\nu+\rho_\ell)$. 
Observe that $c(T)$ is a partition if and only if $c(T)+\rho_\ell$ is a
strict partition, in which case $c(T)+\rho_\ell=w(\nu+\rho_\ell)$ implies
that $w=1$.

For any tableau $T$, let $T_{\geq r}$ denote the subtableau of $T$
formed by the entries in columns $r$ and higher, and define $T_{>r}$
and $T_{<r}$ similarly.  We say that a pair $(w,T)$ is {\em bad} if
$c(T_{\geq r})$ is not a partition for some $r$. Let $Y$ denote the
set of bad pairs indexing the sum (\ref{ceq}), and define a sign
reversing involution $\iota : Y \to Y$ as follows. Given $(w,T)\in Y$,
choose $r$ maximal such that $c(T_{\geq r})$ is not a partition, and
let $j$ be minimal such that $c_j(T_{\geq r}) < c_{j+1}(T_{\geq
r})$. Call an entry $j$ (resp.\ $j+1$) in $T$ {\em free} if there is
no $j+1$ (resp.\ $j$) in its column.  Let $T'$ denote the filling of
$\la/\mu$ obtained from $T$ by replacing all free $j$'s (resp.\
$(j+1)$'s) that lie in $T_{<r}$ with $(j+1)$'s (resp.\ $j$'s), and
then arranging the entries of each row in weakly increasing order.
Since $c(T_{>r})$ is a partition, we deduce that $T$ contains a single
entry $j+1$ in column $r$, and no $j$ in column $r$, while
$c_j(T_{\geq r})+1 = c_{j+1}(T_{\geq r})$. It follows easily from this
that $T'$ is a tableau. We define $\iota(w,T) = (\epsilon_jw,T')$,
where $\epsilon_j$ denotes the transposition $(j,j+1)$. Since
$\epsilon_j c(T_{<r}) = c(T'_{<r})$ and $\epsilon_j(c(T_{\geq r}) +
\rho_\ell) = c(T_{\geq r})+\rho_\ell$, while $T_{\geq r}$ coincides
with $T'_{\geq r}$, it follows that $\epsilon_j(c(T)+\rho_\ell) =
c(T')+\rho_\ell$ and $\iota(w,T)\in Y$. We conclude that the bad pairs
can be cancelled from the sum (\ref{ceq}).

The above argument proves that $c_{\mu\nu}^{\la}$ is equal to the 
number of tableaux $T$ of shape $\la/\mu$ and content $\nu$ such that
$T_{\geq r}$ is a partition for each $r$. This is one among many 
equivalent forms of the {\em Littlewood-Richardson rule}.

\subsection{Duality involution}

Let $v_r = U_{(1^r)}$ for $r \geq 1$, $v_0 = 1$, and $v_r=0$ for
$r<0$.  By expanding the determinant $U_{(1^r)} = \det(u_{1+j-i})_{1
\leq i,j \leq r}$ along the first row, we obtain the identity
\begin{equation}
\label{dualid}
v_r - u_1 v_{r-1} + u_2v_{r-2} - \cdots +(-1)^r u_r = 0.
\end{equation}
Define a ring homomorphism $\omega:\A \to \A$ by setting 
$\omega(u_r) = v_r$ for every integer $r$. For any integer sequence
$\al$, let $v_\al=\prod_i v_{\al_i}$, and for any partition $\la$, set
\[
V_\la = \omega(U_\la) = \prod_{i<j}(1-R_{ij})\, v_\la.
\]
We deduce from (\ref{prule}) that the $V_\la$ satisfy the Pieri rule
\begin{equation}
\label{Vpieri}
v_p\cdot V_\la = \sum_{\la\xrightarrow{p}\mu} V_\mu.
\end{equation}
On the other hand, the Littlewood-Richardson rule easily implies that 
\begin{equation}
\label{dualpieriA}
U_{(1^p)}\cdot U_\la = \sum_\mu U_\mu
\end{equation}
summed over all partitions $\mu\supset\la$ such that $\mu/\la$ is a
{\em vertical $p$-strip}.  It follows from (\ref{Vpieri}),
(\ref{dualpieriA}), and induction on $\la$ that $V_\la = U_{\la'}$ for
each $\la$. Here $\la'$ denotes the partition which is conjugate to
$\la$, i.e.\ such that $\la'_i=\#\{j\ |\ \la_j \geq i\}$ for all $i$.
In particular, the equality $\omega(U_\la)=U_{\la'}$ proves that
$\omega$ is an involution of $\A$, a fact that can also be deduced
from (\ref{dualid}).

\subsection{Cauchy identities and skew Schur polynomials}

Define a new $\Z$-basis $t_\la$ of $\A$ by the transition equations
\begin{equation}
\label{Utf}
U_\la = \sum_\mu K_{\la,\mu}\, t_\mu.
\end{equation}
In other words, the transition matrix $M(U,t)$
between the bases $U_\la$ and $t_\la$ of $\A$ is defined to be
the lower unitriangular Kostka matrix $K$. Then 
$A:= M(t,U) = K^{-1}$ and $B:=M(u,U)=K^t$. We have
\begin{gather*}
\sum_\la t_\la \otimes u_\la = 
\sum_{\la,\mu,\nu}A_{\la\mu}B_{\la\nu}\,U_\mu \otimes U_\nu \\
= \sum_{\la,\mu,\nu}A_{\mu\la}^tB_{\la\nu}\,U_\mu\otimes U_\nu
= \sum_\mu U_\mu \otimes U_\mu
\end{gather*}
in $\A\otimes_\Z \A$, where the above sums are either formal or 
restricted to run over partitions of a fixed integer $n$.
We deduce the Cauchy identity
\begin{equation}
\label{cf}
\sum_\la U_\la \otimes U_\la = \sum_\la t_\la \otimes u_\la
\end{equation}
and, by applying the automorphism $1\otimes\omega$ to (\ref{cf}), the
 dual Cauchy identity
\begin{equation}
\label{dualcf}
\sum_\la U_\la \otimes V_{\la} = \sum_\la t_\la \otimes v_\la.
\end{equation}

For any skew diagram $\la/\mu$, define the {\em skew Schur polynomial}
$U_{\la/\mu}$ by generalizing equation (\ref{Utf}):
\[
U_{\la/\mu} := \sum_\nu  K_{\la/\mu,\nu}\, t_\nu.
\]
We have the following computation in the ring $\A\otimes_\Z \A\otimes_\Z \A$.
\begin{align*}
\sum_{\mu,\nu} U_\mu\otimes U_\nu\otimes U_\mu U_\nu
&= \sum_{\mu,\nu} U_\mu\otimes t_\nu\otimes U_\mu u_\nu
= \sum_{\la,\mu,\nu} U_\mu\otimes t_\nu\otimes K_{\la/\mu,\nu} U_\la\\
&= \sum_{\la,\mu} U_\mu\otimes U_{\la/\mu}\otimes U_\la.
\end{align*}
By comparing the coefficient of $U_\mu\otimes U_\nu\otimes U_\la$
on either end of the previous equation, we obtain
\begin{equation}
\label{skeq}
U_{\la/\mu} = \sum_\nu c_{\mu\nu}^\la\, U_\nu
\end{equation}
where the coefficients $c_{\mu\nu}^\la$ are the same as the ones in
(\ref{scnsts}). Since $\omega(U_\la) = U_{\la'}$ implies the identity
$c_{\mu\nu}^\la = c_{\mu'\nu'}^{\la'}$, we deduce from (\ref{skeq})
that
\begin{equation}
\label{skeqdual}
\omega(U_{\la/\mu}) = U_{\la'/\mu'}.
\end{equation}

\section{Symmetric functions}

\subsection{Initial definitions}
Let $x=(x_1,x_2,\ldots)$ be an infinite sequence of commuting
variables.  For any composition $\al$ we set
$x^\al=\prod_ix_i^{\al_i}$. Given $k\geq 0$, let $\Lambda^k$ denote
the abelian group of all formal power series $\sum_{|\al|=k} c_\al
x^\al\in \Z[[x_1,x_2,\ldots]]$ which are invariant under any
permutation of the variables $x_i$. The elements of $\Lambda^k$ are
called homogeneous symmetric functions of degree $k$, and the graded
ring $\Lambda=\bigoplus_{k\geq 0}\Lambda^k$ is the ring of 
symmetric functions.

For each partition $\la$ of $k$, we obtain an element $m_\la\in
\Lambda^k$ by symmetrizing the monomial $x^\la$. In other words,
$m_\la(x) = \sum_\al x^\al$ where the sum is over all distinct
permutations $\al= (\al_1,\al_2,\dots)$ of 
$\la=(\la_1,\la_2,\ldots)$. We call $m_\la$ a {\em monomial symmetric
function}. The definition implies that if $f = \sum_\al c_\al
x^\al\in \Lambda^k$, then $f = \sum_\la c_\la m_\la$. It follows that
the $m_\la$ for all partitions $\la$ of $k$ (respectively, for all
partitions $\la$) form a $\Z$-basis of $\Lambda^k$ (respectively, 
of $\Lambda$).

Let $h_r=h_r(x)$ denote the $r$-th {\em complete symmetric function},
defined by 
\[
h_r(x) = \sum_{\la\, :\, |\la|=r} m_\la(x) =  
\sum_{i_1\leq\cdots \leq i_r}x_{i_1}\cdots x_{i_r}. 
\]
We have the generating function equation
\begin{equation}
\label{gfe}
H(t)= \sum_{r=0}^\infty h_r(x)t^r = \prod_{i=1}^\infty(1-x_it)^{-1}.
\end{equation}
Let $h_\al = \prod_i h_{\al_i}$ for any integer sequence $\al$.

There is a unique ring homomorphism $\phi:\A \to \Lambda$ 
defined by setting $\phi(u_r)=h_r$ for every $r\geq 0$. For
any integer sequence $\al$, the {\em Schur function} $s_\al$ is
defined by $s_\al = \phi(U_\al)$. We have
\[
s_\al = \prod_{i<j} (1-R_{ij})\, h_\al = 
\det(h_{\al_i+j-i})_{i,j}.
\]

\subsection{Reduction and tableau formulas}
\label{tab}

Let $y=(y_1,y_2,\ldots)$ be a second sequence of variables, choose
$n\geq 1$, and set $x^{(n)}=(x_1,\ldots,x_n)$. It follows easily
from equation (\ref{gfe}) that for any integer $p$, 
\[
h_p(x^{(n)},y) = \sum_{i=0}^p h_i(x_n)\,h_{p-i}(x^{(n-1)},y). 
\]
Therefore, for any integer vector $\nu$, we have
\[
h_\nu(x^{(n)},y) = \sum_{\al\geq 0}
h_{\al}(x_n)\,h_{\nu-\al}(x^{(n-1)},y) = \sum_{\al\geq 0}
x_n^{|\al|}\,h_{\nu-\al}(x^{(n-1)},y)
\]
summed over all compositions $\al$. If $R$ denotes any raising operator
and $\la$ is any partition, we obtain
\begin{equation}
\label{gieq}
R\, h_\la(x^{(n)},y) = \sum_{\al\geq 0}
x_n^{|\al|}\,h_{R\la-\al}(x^{(n-1)},y) = \sum_{\al\geq 0}
x_n^{|\al|}\,R\,h_{\la-\al}(x^{(n-1)},y).
\end{equation}

Since $s_\la =  \prod_{i<j} (1-R_{ij})\, h_\la$, we deduce from  
(\ref{gieq}) that
\[
s_\la(x^{(n)},y) = \sum_{\al\geq 0}
x_n^{|\al|}s_{\la-\al}(x^{(n-1)},y) = \sum_{p=0}^\infty
x_n^p \sum_{|\al|=p}s_{\la-\al}(x^{(n-1)},y).
\]
Applying Lemma \ref{mirrors}, we obtain the reduction formula
\begin{equation}
\label{reduction}
s_\la(x^{(n)},y) = \sum_{p=0}^\infty x_n^p 
\sum_{\mu\xrightarrow{p}\la} 
s_\mu(x^{(n-1)},y).
\end{equation}
Repeated application of the reduction equation (\ref{reduction})
results in
\begin{equation}
\label{reductionT}
s_\la(x^{(n)},y) = \sum_{\mu\subset\la}  s_\mu(y)
\sum_{T \, \text{on} \, \la/\mu} x^{c(T)}
\end{equation}
where the first sum is over partitions $\mu\subset\la$ and the second
over all tableau $T$ of shape $\la/\mu$ with entries at most $n$. As $n$
is arbitrary, equation (\ref{reductionT}) holds with
$x=(x_1,x_2,\ldots)$ in place of $x^{(n)}$. It follows that
\[
s_\la(x,y) = \sum_{\mu\subset\la}  s_\mu(y)
\sum_{T \, \text{on} \, \la/\mu} x^{c(T)}
\]
where the second sum is over all tableau $T$ of shape $\la/\mu$.
Substituting $y=0$ proves Littlewood's tableau formula
\begin{equation}
\label{tf}
s_\la(x) = \sum_{T \, \text{on} \, \la}x^{c(T)} = 
\sum_\mu K_{\la,\mu}\,m_\mu(x).
\end{equation}
From (\ref{tf}) we deduce immediately that the $s_\la$ for $\la$ a 
partition form a $\Z$-basis of $\Lambda$, and comparing with 
(\ref{Utf}) shows that $\phi(t_\la) = m_\la$. 
It follows that the functions
$h_\la$ for $\la$ a partition also form a $\Z$-basis of $\Lambda$.

\subsection{Duality and Cauchy identities}

Let $e_r=e_r(x)$ denote the $r$-th {\em elementary 
symmetric function} in the variables $x$, so that 
\[
e_r(x) = m_{(1^r)}(x) = \sum_{i_1<\cdots < i_r}x_{i_1}\cdots x_{i_r}.
\]
The generating function $E(t)$ for the $e_r$ satisfies
\[
E(t)= \sum_{r=0}^\infty e_r(x)t^r = \prod_{i=1}^\infty(1+x_it).
\]
Since $E(t)H(-t)=1$, we obtain 
\begin{equation}
\label{dualid2}
e_r - h_1 e_{r-1} + h_2e_{r-2} - \cdots +(-1)^r h_r = 0
\end{equation}
for each $r\geq 1$. For any integer sequence $\al$, we set 
$e_\al = \prod_i e_{\al_i}$.

By comparing equations (\ref{dualid}) and (\ref{dualid2}), we deduce
that $\phi(v_r)=e_r$ for each $r$, and hence $\phi(v_\la) = e_\la$ and
$\phi(V_\la) = s_{\la'}$. The duality involution on $\A$ transfers to
an automorphism $\omega:\Lambda \to \Lambda$ which sends $h_\la$ to
$e_\la$ and $s_\la$ to $s_{\la'}$, for each partition $\la$. We deduce
that the $e_\la$ form another $\Z$-basis of
$\Lambda$.  Moreover, by applying $\phi$ to (\ref{cf}) and
(\ref{dualcf}), we obtain the usual form of the Cauchy identities
\[
\sum_\la s_\la(x)s_\la(y) = \sum_\la m_\la(x) h_\la(y)
=\prod_{i,j}\frac{1}{1-x_iy_j}
\]
and 
\[
\sum_\la s_\la(x)s_{\la'}(y) = \sum_\la m_\la(x) e_\la(y)
=\prod_{i,j}(1+x_iy_j)
\]
where the sums are taken over all partitions $\la$.

\subsection{Skew Schur functions}

Define the {\em skew Schur functions} $s_{\la/\mu}$ by 
\[
s_{\la/\mu}(x) = \phi(U_{\la/\mu}) = \sum_{\nu} K_{\la/\mu,\nu}\, m_\nu(x)
= \sum_{T \, \text{on} \, \la/\mu} x^{c(T)}.
\]
Equation (\ref{reductionT}) then implies that
\begin{equation}
\label{seq1}
s_\la(x,y) = \sum_{\mu\subset \la}s_{\la/\mu}(x)s_\mu(y)
= \sum_{\mu\subset \la}s_\mu(x)s_{\la/\mu}(y).
\end{equation}
Applying the operator $\prod_{i<j}(1-R_{ij})$ to both sides of the 
equation
\[
h_\la(x,y) = \sum_{\al\geq 0} h_\al(x) h_{\la-\al}(y)
\] 
gives 
\begin{equation}
\label{seq2}
s_\la(x,y) = 
\sum_{\al\geq 0} h_\al(x)s_{\la-\al}(y).
\end{equation}
Since $h_\al = \sum_\mu K_{\mu,\al}s_\mu$, comparing (\ref{seq1}) with
(\ref{seq2}) proves that
\begin{equation}
\label{genmir}
s_{\la/\mu} = \sum_{\al \geq 0} K_{\mu,\al}s_{\la-\al}.
\end{equation}
Observe that (\ref{genmir}) is a generalization of the second identity
in Lemma \ref{mirrors}.

Using Lemma \ref{altlem}(b) in (\ref{seq2}), we obtain that
\begin{equation}
\label{step}
s_\la(x,y) = \sum_\mu s_\mu(y) \sum_{w\in S_\ell}(-1)^w
h_{\la+\rho_\ell - w(\mu + \rho_\ell)}(x)
\end{equation}
where the first sum is over all partitions $\mu$ and $\ell =
\ell(\la)$.  Equating the coefficients of $s_\mu(y)$ in (\ref{seq1})
and (\ref{step}) proves the following generalization of the
Jacobi-Trudi identity (\ref{def2}):
\begin{equation}
\label{skewjt}
s_{\la/\mu} = \sum_{w\in S_\ell}(-1)^w
h_{\la+\rho_\ell - w(\mu + \rho_\ell)} = 
\det(h_{\la_i-\mu_j+j-i})_{i,j}. 
\end{equation}
By applying the involution $\omega$ to (\ref{skewjt}) and using
(\ref{skeqdual}), we derive the dual equation
\[
s_{\la'/\mu'} =  \det(e_{\la_i-\mu_j+j-i})_{i,j}. 
\]

\subsection{The classical definition of Schur polynomials}

In this section we fix $n$, the number of variables, and work with
integer vectors and partitions in $\Z^n$. Let $x= (x_1,\ldots,x_n)$
and set $\rho=\rho_n= (n-1,\ldots,1,0)$. For each $\al\in\Z^n$, define
\[
A_\al = \sum_{w\in S_n}(-1)^w x^{w(\al)} = 
\det(x_i^{\al_j})_{1\leq i,j \leq n}
\]
and set $\tilde{s}_\al(x) = A_{\al+\rho}/A_\rho$. Consider the
$\Z$-linear map $\A\to\Z[x_1,\ldots,x_n]$ sending $U_\la$ to
$A_{\la+\rho}$ for any partition $\la$ with $\ell(\la)\leq n$, and to
zero, if $\ell(\la)>n$. It follows from Lemma \ref{altlem}(b) that
this map sends $U_\al$ to $A_{\al+\rho}$ for any composition
$\al\in\Z^n$. Lemma \ref{mirrors} therefore implies that for any
partition $\la\in \Z^n$ and integer $r\geq 0$, we have
\begin{equation}
\label{altmrs}
\sum_{\al\geq 0} A_{\la+\al+\rho} = \sum_{\la\xrightarrow{r}\mu}A_{\mu+\rho}
\end{equation}
where the sums are over compositions $\al\geq 0$ with $|\al|=r$ and
$\ell(\al)\leq n$ and partitions $\mu$ with $\la\xrightarrow{r}\mu$
and $\ell(\mu)\leq n$. Furthermore, we have
\begin{align*}
A_{\la+\rho}\,h_r(x) 
& = \sum_{w\in S_n}(-1)^w\sum_{\al\geq 0\, :\ |\al|=r} x^{w(\la+\rho)+\al} \\ 
& = \sum_{w\in S_n}(-1)^w\sum_{\al\geq 0\, :\ 
|\al|=r} x^{w(\la+\rho)+w(\al)} \\
& = \sum_{\al\geq 0\, :\, |\al|= r} A_{\la+\al+\rho} =
\sum_{\la\xrightarrow{r}\mu}A_{\mu+\rho},
\end{align*}
by (\ref{altmrs}). Now divide by $A_\rho$ to
deduce that 
\begin{equation}
\label{Aeq}
\tilde{s}_{\la}(x) \, h_r(x) = \sum_{\la\xrightarrow{r}\mu}
\tilde{s}_{\mu}(x).
\end{equation}
Applying (\ref{Aeq}) with $\la=0$ gives $\tilde{s}_r(x) = h_r(x)$, for
every $r\geq 1$. Since the $\tilde{s}_\la(x)$ satisfy the Pieri rule,
it follows by induction on $\la$ as in \S \ref{pieriruleA} that
\[
\tilde{s}_\la(x) = \prod_{i<j}(1-R_{ij})\, h_\la(x) = s_\la(x)
\]
for each partition $\la$ of length at most $n$. We have thus
proved equation (\ref{def1}).

\end{document}